\documentclass[11pt]{amsart}
\usepackage{amssymb,amsmath,amsthm}
\usepackage{a4wide}
\usepackage{graphicx}

\usepackage{color}
\usepackage{mathrsfs}
\usepackage{latexsym}
\usepackage{bbm}
\usepackage{hyperref}
\usepackage{mathtools}

\numberwithin{equation}{section}

\theoremstyle{plain}
\begingroup
\newtheorem{theorem}{Theorem}[section]
\newtheorem{lemma}[theorem]{Lemma}

\endgroup

\theoremstyle{definition}
\begingroup

\newtheorem{remark}[theorem]{Remark}

\endgroup

\newcommand{\F}{\mathcal{F}}

\newcommand{\N}{\mathbb{N}}

\newcommand{\R}{\mathbb{R}}

\newcommand{\G}{\mathcal{G}}
\newcommand{\ffi}{\varphi}

\newcommand{\supp}{\mathrm{supp}\;}
\newcommand{\ud}{\,\mathrm{d}}
\def\H{\mathscr{H}}

\def\J{\mathcal{J}}
\newcommand{\weakly}{\rightharpoonup}
\newcommand{\weakstar}{\overset{\ast}{\rightharpoonup}}

\newcommand{\cc}{\mathrm{c}}
\newcommand{\CC}{\mathrm{C}}
\newcommand{\B}{\mathcal D}
\newcommand{\h}{\mathscr{H}}

\newcommand{\Aa}{\alpha}

\definecolor{ddorange}{rgb}{1,0.5,0}
\definecolor{ddcyan}{rgb}{0,0.2,1.0}

\title[Parabolic Riesz flows] {Parabolic $\alpha$-Riesz flows and  limit cases $\alpha\to 0^+$, $\alpha\to d^-$}

\author[L. De Luca]
{L. De Luca}
\address[Lucia De Luca]{Istituto per le Applicazioni del Calcolo ``M. Picone'' IAC-CNR, Via dei Taurini 19, I-00185 Roma, Italy}
\email[L. De Luca]{lucia.deluca@cnr.it}

\author[M. Morini]
{M. Morini}
\address[Massimiliano Morini]{Dipartimento di Scienze Matematiche, Fisiche e Informatiche, Universit\`a degli Studi di Parma, Parma, Italy}
\email[M. Morini]{massimiliano.morini@unipr.it}

\author[M. Ponsiglione]
{M. Ponsiglione}
\address[Marcello Ponsiglione]{Dipartimento di Matematica ``Guido Castelnuovo'', Sapienza Universit\`a di Roma, Piazzale Aldo Moro 2, I-00185 Roma, Italy
}
\email[M. Ponsiglione]{ponsigli@mat.uniroma1.it}

\author[E. Spadaro]
{E. Spadaro}
\address[Emanuele Spadaro]{Dipartimento di Matematica ``Guido Castelnuovo'', Sapienza Universit\`a di Roma, Piazzale Aldo Moro 2, I-00185 Roma, Italy
}
\email[E. Spadaro]{spadaro@mat.uniroma1.it}

\begin{document}
 \maketitle

\begin{abstract}
In this paper we introduce the notion of  parabolic $\alpha$-Riesz  flow, for $\alpha\in(0,d)$,  extending the notion of $s$-fractional   heat flows to negative values of the parameter $s=-\frac{\alpha}{2}$. Then, we determine the limit behaviour of these gradient flows as $\alpha \to 0^+$ and $\alpha \to d^-$. 

To this end we  provide a preliminary $\Gamma$-convergence expansion  for the  Riesz interaction energy functionals. Then we apply  abstract stability results for  uniformly $\lambda$-convex functionals which  guarantee that $\Gamma$-convergence commutes with the  gradient flow structure.  

\vskip5pt
\noindent
\textsc{Keywords:}  Gagliardo seminorms; $\Gamma$-convergence; Fractional heat flow
\vskip5pt
\noindent
\textsc{AMS subject classifications: }
49J45    
35R11   
35K20  
\end{abstract}
\setcounter{tocdepth}{1} 
\tableofcontents
\section*{Introduction}
Fractional versions of local operators attracted much attention in the last decades. For  $s\in(0,1)$ the so-called  $s$-fractional Laplacians and 
$s$-fractional perimeters represent natural and interesting nonlocal counterparts of the classical Laplacian operator and Euclidean  perimeter. 
It is also well known \cite{Landkof, Stein} that for negative values of the parameter $s$, and specifically for $s=- \frac\alpha 2$ with $\alpha\in(0,d)$ ($d$ being the dimension of the ambient space)\,, the $\alpha$-Riesz potentials represent the natural continuation of the one-parameter family of $s$-fractional Laplacians (as well as of $s$-fractional perimeters and curvatures \cite{CDNP}) for negative values of $s$.      

Much analysis has been devoted to studying the limit cases of $s$-fractional perimeters and  Laplacians as $s\to 0^+$ and $s\to 1^-$\,; remarkably, as $s\to 1^-$ one recovers the local canonical objects, while the limit $s\to 0^+$ is somehow more degenerate. The analogous analysis for  Riesz type operators seems to be less investigated, with recent results only for what concerns geometric flows \cite{CDNP}. This paper gives a further contribution in this direction, introducing the notion of 
{\it parabolic $\alpha$-Riesz flows} and
studying the limit cases as $\alpha\to 0^+$ and $\alpha\to d^-$. 

We will now revisit the literature concerning the limit analysis of $s$-fractional operators as  $s\to 0^+$
(referring  the  interested reader to \cite{ADPM, BBM1, BBM2, BPS,  CDKNP, D,  LS, P, MRT} for the case  $s\to 1^-$).

In \cite{MS} the authors show that, as $s\to 0^+$\,, the squared $s$-fractional Gagliardo seminorms multiplied by $s$ pointwise converge to (a multiple of) the squared $L^2$ norm.
A $\Gamma$-convergence expansion (as $s\to 0^+$) with respect to the  $L^2$ topology has been derived in \cite{CDKNP}. 
The first term of such an expansion is, according with \cite{MS}, the squared $L^2$ norm divided by $s$\,, whereas the next order term can be understood as a kind of  $0$-fractional Gagliardo seminorm. The first variation of the zero-order $\Gamma$-limit is trivial, while the first variation of the first order $\Gamma$-limit gives an interesting elliptic operator, referred to as {\it $0$-fractional Laplacian},  which is a slight variant of the logarithmic Laplacian introduced in \cite{CW}.
Still in  \cite{CDKNP}, the minimizing movements approach was used as a bridge between the $\Gamma$-convergence analysis and the stability of the corresponding fractional parabolic flows:  
 After scaling the time by $s$\,, the $s$-fractional parabolic flows converge, as $s\to 0^+$\,, 
to the degenerate parabolic flow of the squared $L^2$ norm (with respect to the $L^2$ metric, itself); this is nothing but a trivial ODE corresponding to the exponential decay of the initial datum. If, instead of scaling in time, we renormalize the parabolic flow by subtracting the forcing leading term (corresponding to the first variation of the squared $L^2$ norm, scaled by $s$), we end up with a much more interesting dynamics: 
the heat flow governed by the $0$-fractional Laplacian, namely,
the parabolic flow of the first order $\Gamma$-limit of the $s$-fractional Gagliardo seminorms.

A first task of this paper is to provide a similar analysis as $s\to 0^-$. As already explained, for negative values of $s$ the analogue of fractional Laplacians is given by $\Aa$-Riesz potentials \cite{Landkof, Stein}, with $\Aa=-2s$\,. The first step in our analysis is to formalize such an analogy introducing, for negative values of $s$,  explicit counterparts of $s$-fractional Gagliardo  seminorms,  Laplacians, and the corresponding parabolic flows, referred to as  parabolic $\alpha$-Riesz flows. 
 Such a track  has already been  pursued  in  the geometric framework \cite{CDNP}; there,   suitable notions of $\alpha$-Riesz perimeters, curvatures and geometric flows  have been introduced and analyzed  as $\alpha \to 0^+$, leading to the geometric flow of the $0$-fractional perimeter introduced in \cite{DNP}.  (In fact, such a  flow also coincides with the limit as $s\to 0^+$ of the $s$-fractional mean curvature flows \cite{CDNP}.)

As in \cite{CDKNP, CDNP}, our approach is purely variational, in the sense that our stability analysis relies  on the gradient flow structure of the parabolic $\alpha$-Riesz flows. 
We consider an open bounded set $\Omega\subset\R^d$ and,
  for every $\alpha\in (0,d)$\,, we define the {\it $\Aa$-Riesz functionals} $\J^\Aa$ on $L^2(\Omega)$ as
 \begin{equation}\label{intro_energy}
 \J^\Aa(u):=-\int_{\Omega}\int_{\Omega}\frac{u(x)u(y)}{|x-y|^{d-\Aa}}\ud y\ud x\,.
  \end{equation}
 We stress that the first variation of $\J^\Aa$ gives back the opposite of the $\alpha$-Riesz potentials $(-\Delta)^{-\frac{\Aa}{2}}$ \cite{Stein, Landkof}. 
  In terms of a more genuine variational viewpoint, $\J^\Aa$ can be understood  as a continuous version of the $XY$ energy functional with long-range interactions and real-valued  spin variable. 
Oversimplifying, in classical $XY$ scalar systems  the spin variable $u$ takes values $\pm 1$, and, for nearest neighbor interactions, the energy can be seen as a discrete perimeter of the phase $\{u=1\}$\,, as well as a discrete Dirichlet energy for functions defined on a lattice and taking only two values. Clearly, long-range variants of the $XY$ model are naturally related with nonlocal perimeters; specifically, with fractional type perimeters, when the interaction kernel is non-integrable, and with Riesz type perimeters, when the kernel is locally integrable as in \eqref{intro_energy}. 
Although such analogies are part of the motivations for our study, here we will not push further considerations within such a statistical mechanics framework; in fact, we focus our analysis on real-valued functions, adopting tools and terminology of fractional seminorms and operators. 
  
  Now we describe in detail our results, starting from the asymptotic behavior of $\J^{\Aa}$ as $\Aa\to 0^+$\,. 
  In Theorem \ref{gammafs} we prove that the functionals $-\Aa\J^\Aa$ $\Gamma$-converge as $\Aa\to 0^+$ to a multiple of $\|\cdot\|^2_{L^2(\Omega)}$ with respect to the weak $L^2$ topology. Some considerations about the signs of the energy functionals are in order: the functionals $-\J^{\Aa}$ penalize more and more, as $\Aa\to 0^+$\,, functions with large $L^2$ norm, which, instead, are favored by the functionals $\J^{\Aa}$\,. Indeed, also after the natural scaling, the functionals $\Aa\J^\Aa$ would $\Gamma$-converge to the relaxation (with respect to the weak $L^2$ topology) of $-\|\cdot\|_{L^2(\Omega)}^2$, which is constantly equal to $-\infty$ (see Remark \ref{secambiosegno}). On the other hand, renormalizing the (otherwise asymptotically ill-posed) family of functionals $\J^\Aa$\,, by setting $\widehat{\J}^{\Aa}(u):=\J^{\Aa}(u) + \frac{d\omega_d}{\alpha} \|u\|_{L^2(\Omega)}^2$\,, it turns out that the energies $\widehat{\J}^{\Aa}$ penalize oscillations. In fact, in Theorem \ref{thm:stozerofo} we prove that the $\Gamma$-limit of $\widehat\J^\Aa$ (as $\alpha\to 0^+$), with respect to the strong $L^2$ topology, 
  is a nonlocal coercive Dirichlet type energy functional $\widehat{\J}^0$\,, which is nothing but 
 the $\Gamma$-limit  derived in \cite{CDKNP} 
 for the renormalized $s$-fractional Gagliardo seminorms (as $s\to 0^+$).  
 Similarly to what observed above, here there is only one option in choosing the signs in order to deal with coercive functionals: reversing all the  signs, one would end up with a kind of relaxation of  $-\widehat{\J}^0$\,, 
  which would give back once again $-\infty$ (Remark \ref{secambiosegno2}).  Similar considerations can be done in the setting of nonlocal geometric flows, but,  trying to draw such a parallelism, a relevant difference emerges: 
   in the geometrical context    $u$ is a characteristic function, so that oscillations relaxing the energy to $-\infty$ are automatically prevented. A relevant consequence is that, 
  contrarily to the case discussed above, in the geometric setting, the right sign is the negative one, since  
   the functionals $\Aa\J^\Aa$ computed on characteristic functions, together with their limit,  provide  a true perimeter, whose first variation is a nonlocal curvature satisfying natural comparison principles \cite{CDNP}. 
   These different paths of geometrical and linear framework rejoin  when looking at the next order, where
   we fix the negative sign for the  $\alpha$-fractional Laplacians. 

 In this framework, we introduce the  parabolic $\Aa$-Riesz flows as the $L^2$ gradient flows of the functionals $\J^{\Aa}$\,, scaled or renormalized according with the $\Gamma$-convergence analysis. As mentioned above, the first variation of the functional $-\J^{\Aa}$ (for $\Aa\in (0,d)$) is already known in the literature  \cite{Stein, Landkof} under the name of $\Aa$-Riesz potential $(-\Delta)^{-\frac{\Aa}{2}}$\,. 
 Therefore, the fractional parabolic flows considered here are nothing but fractional  heat flows governed by $s$-fractional Laplacians for negative values of the fractional parameter $s\in (-\frac d 2,0)$.

As one may expect from the static asymptotic analysis,  scaling the time by $\Aa$\,, the $\Aa$-parabolic flows governed by the operator $(-\Delta)^{-\frac{\Aa}{2}}$ converge (Theorem \ref{convheat0ordsto0})  as $\Aa\to 0^+$, to a degenerate ODE describing the exponential decay of the initial datum, namely, 
to the gradient flow of  the (up to a prefactor) $\Gamma$-limit $\|\cdot\|^2_{L^2(\Omega)}$; this result agrees with   the analysis  of $s$-fractional heat flows as $s\to 0^+$ done in \cite{CDKNP}.  In analogy with the geometric case of fractional curvature flows \cite{CDNP} and following the discussions on the signs done above, one could wonder what happens reversing all the signs. In fact,  the parabolic flows governed by $- \alpha (-\Delta )^{-\frac{\Aa}{2}} $ is actually well defined; however, as explained in Remark \ref{anchequestovale}, our variational methods based on $\Gamma$-convergence  do not provide its convergence (as $\alpha\to 0^+$) to the parabolic flow  $u_t=2u$. The latter describes  the exponential growth of the initial datum, and is consistent with the corresponding geometric flows, where balls tend to expand (instead of shrinking) faster and faster as $\alpha \to 0^+$. 


Following the first order $\Gamma$-convergence analysis, the gradient flow of  the functional $\widehat\J^{\Aa}$\, is a  parabolic flow governed by $-(-\Delta)^{-\frac{\Aa}{2}}$ plus the forcing leading linear term (diverging as $\alpha\to 0^+$). This flow converges (Theorem \ref{convheat1ordsto0}), as $\alpha\to 0^+$, to the gradient flow of  $\widehat\J^{0}$\,, namely, to the parabolic flow governed by  the $0$-fractional Laplacian. 
Again, this result agrees with the analogous analysis done for $s\to 0^+$ in \cite{CDKNP}. 
In fact,
 the proofs of the stability results above are obtained as an application of the abstract approach developed in \cite{CDKNP}; loosely speaking, the latter consists in showing that uniform $\lambda$-convexity and $\lambda$-positivity on  an Hilbert setting is a sufficient condition to guarantee that $\Gamma$-convergence of the underlying functionals commutes with the gradient flow structure. We refer to Subsection \ref{astrrisu} for a short recap on such an abstract setting.

 The same approach, namely, first analyzing the asymptotic behavior of the functionals $\J^{\Aa}$ as $\Aa\to d^-$ and then using the abstract result in \cite{CDKNP} to derive the convergence of the corresponding parabolic flows, allows to study also the limit as $\Aa\to d^-$ for the parabolic $\Aa$-Riesz flows. In such a case, the analysis is simpler. Indeed, the interaction kernel  in \eqref{intro_energy} converges to $1$ (strongly in $L^1_{\mathrm{loc}}(\R^d)$), whence, for every  $u\in L^2(\Omega)$\,,  we deduce that $\J^{\Aa}(u)$ converges (as $\Aa\to d^-$) to   $\J^{d}(u):=-\big(\int_\Omega u\big)^2$\,. Actually, in Theorem \ref{thm:sto-d2}, we prove that $\J^{\Aa_n}(u^n)\to \J^d(u)$ for every $\Aa_n\to d^{-}$ and for every $u^n\weakly u$ in $L^2(\Omega)$\,, which provides in particular the $\Gamma$-convergence of the functionals $\pm\J^{\alpha}$ to $\pm\J^{d}$\,. As a consequence, in Theorem \ref{convheat0ordsto-d2} we prove that the $\Aa$-parabolic flows governed by the operator $\pm (-\Delta)^{-\frac{\Aa}{2}}$ converge (Theorem \ref{convheat0ordsto0}) , as $\Aa\to d^-$, to a degenerate ODE describing now the exponential growth/decay of the average of the initial datum, namely, 
to the gradient flow of  the  $\Gamma$-limit $\mp \big(\int_\Omega u\big)^2$.
Moreover, for every $u^n\weakly u$ in $L^2(\Omega)$ and for every $\alpha_n\to d^-$, the renormalized functionals 
$
\widetilde{\J}^{\Aa_n}(u^n):=\frac{\J^{\Aa_n}(u^n)-\J^{d}(u^n)}{d-\Aa_n}\,
$
converge to the Riesz energy functional with logarithmic kernel 
$$
\widetilde\J^{d}(u):=-\int_{\Omega}\int_{\Omega}u(x)u(y) \log\frac{1}{|x-y|} \ud y\ud x\,.
$$
In particular, as $\Aa\to d^-$\,, the functionals $\pm \widetilde\J^{\Aa}$  $\Gamma$-converge to the functional $\pm \widetilde\J^{d}$ with respect to the weak $L^2$ convergence 
(see Theorem \ref{thm:sto-d2fo}). 
Finally, Theorem \ref{convheat1ordsto-d2} establishes that the gradient flows of $\pm \widetilde{\J}^\Aa$ converge, as $\Aa\to d^-$, to the parabolic flow governed by 
the operator $\pm(-\Delta)^{-\frac d 2} $, namely, by the first variation of $\pm \widetilde\J^{d}$.
 
Summarizing, we have introduced the parabolic $\Aa$-Riesz flows, i.e., the flows governed by the $\Aa$-Riesz potential $(-\Delta)^{-\frac{\Aa}{2}}$\,, for $\Aa\in (0,d)$, and we have analyzed their convergence as $\Aa\to 0^+$ and $\Aa\to d^-$\,. 
Along the way, we have focused on  the interesting class of integral functionals 
\begin{equation}\label{intro_energy2}
-\int_{\Omega}\int_{\Omega}\frac{u(x)u(y)}{|x-y|^{p}}\ud y\ud x\,,
  \end{equation}
representing, for $p\in(0,d)$,  natural extensions of  Gagliardo-type energies. In particular, through rigorous renormalization procedures, we have given a  meaning to    the critical case $p=d$.
It may be of some interest to  extend our methods to   $p\ge d$, adopting suitable  renormalization procedures; in this direction, a popular method is  the so-called 
{\it core radius approach}, already used in \cite{DKP} in the geometric framework. The  general case $p\in\R$, accounting also for negative values of $p$, could deserve further investigations. Ultimately, our analysis invites to a   
methodical comparison between the  one-parameter family \eqref{intro_energy2} of generalized Riesz functionals/potentials and the (generalized) Gagliardo seminorms/fractional Laplacians.

\vskip10pt
\textsc{Acknoledgments:}  LDL and MM are members of the Gruppo Nazionale per l'Analisi Matematica, la Probabilit\`a e le loro Applicazioni (GNAMPA) of the Istituto Nazionale di Alta Matematica (INdAM). The authors have been partially funded by the ERC-STG Grant n. 759229 HiCoS ``Higher Co-dimension
Singularities: Minimal Surfaces and the Thin Obstacle Problem''.
 \vskip10pt

\section{Description of the problem}\label{Sec1}
Let $d\in\N$ and let $\Omega\subset\R^d$ be a bounded and open set  with Lipschitz continuous boundary.
For every $\Aa\in (0,d)$\,, we define the functional $\J^\Aa:L^2(\Omega)\to(-\infty,+\infty)$ as 
\begin{equation}\label{defJs}
\J^\Aa(u):=-\int_{\Omega}\int_{\Omega}\frac{u(x)u(y)}{|x-y|^{d-\Aa}}\ud y\ud x\,.
\end{equation}
We first show that the functionals $\J^\Aa$ are well defined and actually finite in $L^2(\Omega)$. Let $k^{\Aa}:\R^d\to[0,+\infty)$ be the function defined by
\begin{equation}\label{ks}
k^{\Aa}(z):=\frac{1}{|z|^{d-\Aa}}\,.
\end{equation}
Then, $k^\Aa\in L^1_{\mathrm{loc}}(\R^d)$\,. 
Moreover, for every $R>0$ we set 
\begin{equation}\label{ksristr}
k^\Aa_R:=\chi_{B_R(0)} k^\Aa\,;
\end{equation}
it is immediate to check that
\begin{equation}\label{l1ksR}
\|k^{\Aa}_R\|_{L^1(\R^d)}= \frac{d\omega_d}{\Aa}R^{\Aa}\,.
\end{equation}
Now, for any $u\in L^2(\Omega)$ and for every $R>\mathrm{diam}(\Omega)$, we have
\begin{equation}\label{riscr}
\J^\Aa(u)=-\langle \tilde u, \tilde u\ast k^\Aa\rangle_{L^2(\R^d)}=-\langle \tilde u, \tilde u\ast k^\Aa_R\rangle_{L^2(\R^d)}\,,
\end{equation}
where, here and throughout the paper, $\tilde u\in L^2(\R^d)$ is defined by $\tilde u=u$ in $\Omega$ and $\tilde u=0$ in $\R^d\setminus\overline\Omega$\,.

Therefore, by \eqref{riscr} and \eqref{l1ksR}, using H\"older and convolution's Young inequalities, we obtain that for every $u\in L^2(\Omega)$
\begin{equation}\label{welldefJs}
|\J^\Aa(u)|\le \|k^\Aa_R\|_{L^1(\R^d)}\|\tilde u\|^2_{L^2(\R^d)}\le  \frac{d\omega_d}{\Aa}R^{\Aa}\|u\|^2_{L^2(\Omega)}\,,
\end{equation}
which shows that $\J^\Aa$ takes values in $(-\infty,+\infty)$\,.

We highlight that for any $\Aa\in (0,d)$ the Fourier transform $\mathscr{F}[k^\Aa]$ of $k^{\Aa}$ is given by
\begin{equation}\label{fouritra}
\mathscr{F}[k^\Aa](\xi)=2^\Aa\pi^{\frac d 2}\frac{\Gamma(\frac{\Aa}{2})}{\Gamma(\frac{d-\Aa}{2})}k^{d-\Aa}(\xi)\,,
\end{equation}
where $\Gamma$ is the Euler Gamma function, defined by $\Gamma(\beta):=\int_{0}^{+\infty}t^{\beta-1}e^{-t}\ud t$\,, for $\beta>0$\,.
Indeed, let $\ffi:\R^d\to\R$ belong to the Schwarz class. Then, setting $\beta:=d-\Aa$\,, by the very definition of Euler Gamma function and by Fubini Theorem, we have that
\begin{equation}\label{serfaty}
\int_{\R^d}\ffi(x)k^{\Aa}(x)\ud x=\frac{1}{\Gamma(\frac\beta 2)}\int_{0}^{+\infty}\!\ud t \, t^{\frac{\beta}{2}-1} \int_{\R^d}\!\ud x \,\ffi(x)e^{-t|x|^2}\,.
\end{equation}
Moreover, since
\begin{equation*}
\mathscr{F}[e^{-t|\cdot|^2}](\xi):=\int_{\R^d}e^{-t|x|^2}e^{-ix\cdot \xi}\ud x=\pi^{\frac d 2}\frac{e^{-\frac{|\xi|^2}{4t}}}{t^{\frac d 2}},
\end{equation*}
by Parseval identity we obtain
\begin{equation*}
\int_{\R^d}\ffi(x)e^{-t|x|^2}\ud x=\frac{\pi^{\frac d 2}}{(2\pi)^d}\int_{\R^d}\mathscr{F}[\ffi](\xi)\frac{e^{-\frac{|\xi|^2}{4t}}}{t^{\frac d 2}}\ud \xi\,,
\end{equation*}
which, plugged into \eqref{serfaty}, yields
\begin{equation*}
\begin{aligned}
\int_{\R^d}\ffi(x)k^{\Aa}(x)\ud x=&\,\frac{\pi^{\frac d 2}}{(2\pi)^d}\frac{1}{\Gamma(\frac\beta 2)}\int_{0}^{+\infty}\!\ud t \, t^{\frac{\beta-d}{2}-1} \int_{\R^d}
\mathscr{F}[\ffi](\xi)e^{-\frac{|\xi|^2}{4t}}\ud \xi\\
=&\,\frac{\pi^{\frac d 2}}{(2\pi)^d}2^{\Aa}\frac{\Gamma\big(\frac{\Aa}{2}\big)}{\Gamma(\frac{d-\Aa} 2)}\int_{\R^d}{\mathcal{F}[\ffi](\xi)} k^{d-\Aa}(\xi)\ud\xi\,,
\end{aligned}
\end{equation*}
where we have used first the change of variable $t=s|\xi|^2$ and then $u=\frac 1 {4s}$ to get
\begin{equation*}
\begin{aligned}
\int_{0}^{+\infty}t^{\frac{\beta-d}{2} -1}e^{-\frac{|\xi|^2}{4t}}\ud t=&\,|\xi|^{-\Aa}\int_{0}^{+\infty}s^{\frac{\beta-d}{2} -1}e^{-\frac{1}{4s}}\ud s\\
=&\,|\xi|^{-\Aa}2^{d-\beta}\int_{0}^{+\infty}u^{\frac{d-\beta}{2} -1}e^{-u}\ud u\\
=&\, |\xi|^{-\Aa}2^{\alpha}\Gamma\big(\frac{\alpha}{2}\big)\,.
\end{aligned}
\end{equation*}
Therefore, \eqref{fouritra} holds true and hence
the kernels $k^{\Aa}$ are positive definite, i.e.,
\begin{equation}\label{posdef}
\langle v, v\ast  k^{\Aa} \rangle_{L^2(\R^d)}\ge 0\qquad\textrm{for any }v\in L^2(\R^d)\,,
\end{equation}
which, by \eqref{riscr}, yields
\begin{equation}\label{posdef2}
\J^{\Aa}(u)
\le 0\qquad\textrm{for any }u\in L^2(\Omega)\,.
\end{equation}

Moreover, by \eqref{riscr}, using the  symmetry of $k^\Aa_R$, for any $u,v\in L^2(\Omega)$ we have that
\begin{equation}\label{continJs}
\begin{aligned}
|\J^\Aa(u)-\J^\Aa(v)|=&\Big|\langle u-v,(u+v)\ast k^\Aa_R\rangle_{L^2(\R^d)}\Big|\\
\le&\frac{d\omega_d}{\Aa}R^{\Aa} \big(\|u\|_{L^2(\Omega)}+\|v\|_{L^2(\Omega)}\big)\|u-v\|_{L^2(\Omega)}\,,
\end{aligned}
\end{equation}
whence we deduce the continuity of $\J^\Aa$ with respect to the strong convergence in $L^2(\Omega)$\,.

Now, for every $\Aa\in(0,d)$ and for every $u\in L^2(\Omega)$ we set 
\begin{equation}\label{defJshat}
\widehat\J^\Aa(u):=\J^\Aa(u)+\frac{d\omega_d}{\Aa}\|u\|^2_{L^2(\Omega)}\,.
\end{equation}
Let $\B_1:=\{(x,y)\in\R^d\times\R^d\,:\,|x-y|\le 1\}$\,; for every $\Aa\in[0, d)$ we define the functionals $\G^\Aa_1:L^2(\Omega)\to[0,+\infty]$ and $\J^\Aa_1:L^2(\Omega)\to\R$ as
\begin{equation*}
\G^\Aa_1(u):=\frac 1 2\iint_{\B_1}\frac{|\tilde u(x)-\tilde u(y)|^2}{|x-y|^{d-\Aa}}\ud y\ud x\,,\qquad
\J^\Aa_1(u):=-\iint_{(\Omega\times\Omega)\setminus\B_1}\frac{u(x)u(y)}{|x-y|^{d-\Aa}}\ud y\ud x\,.
\end{equation*}
We highlight that for every $\Aa\in(0,d)$ and for every
 $u\in L^2(\Omega)$ it holds
\begin{equation}\label{Jshat_re}
\widehat\J^\Aa(u)=\G^\Aa_1(u)+\J^\Aa_1(u)\,,
\end{equation}
Notice that, by H\"older inequality, for every $\Aa\in[0,d)$ and for every $u\in L^2(\Omega)$
\begin{equation}\label{estJs1}
|\J^\Aa_1(u)|\le \iint_{\Omega\times\Omega}|u(x)u(y)|\ud y\ud x\le \|u\|^2_{L^1(\Omega)}\le |\Omega|\|u\|^2_{L^2(\Omega)}\,,
\end{equation}
and hence $\J^\Aa_1$ is finite on $L^2(\Omega)$ also for $\Aa=0$\,.
In full analogy with \eqref{Jshat_re} we can define the  functional $\widehat\J^\Aa$ also for $\Aa=0$ by setting
\begin{equation}\label{J0hat}
\widehat\J^0(u):=\G^0_1(u)+\J^0_1(u)\,,
\end{equation}
for every $u\in L^2(\Omega)$\,. In view of \eqref{estJs1}, the functional $\widehat\J^0:L^2(\Omega)\to (-\infty,+\infty]$ is well defined.
Notice that $\widehat\J^0$ in \eqref{J0hat} coincides with the functional $\hat F^0$ defined in \cite[formula (1.9)]{CDKNP}.
\begin{remark}\label{DCT}
In view of its very definition \eqref{defJshat} and of
 \eqref{continJs},  the functionals $\widehat{\J}^\Aa$ with $\Aa\in(0,d)$ are continuous with respect to the strong convergence in $L^2(\Omega)$\,.
 Moreover, for every $\Aa\in [0,d)$ and
  for all $u,\, v \in L^2(\Omega)$ we have
\begin{equation}\label{DCT2}
\begin{aligned}
|\J^\Aa_1(u) - \J^\Aa_1(v)| \le &  (\|u\|_{L^1(\Omega)} + \|v\|_{L^1(\Omega)})\|u-v\|_{L^1(\Omega)}
\\
\le & |\Omega| (\|u\|_{L^2(\Omega)} + \|v\|_{L^2(\Omega)})\|u-v\|_{L^2(\Omega)}\, . 
\end{aligned}
\end{equation}
which shows, in particular, that 
the functionals $\widehat{\J}^{0}$ are lower semicontinuous with respect to the strong convergence in $L^2(\Omega)$\,.
\end{remark}
The following result will be used throughout the paper.
\begin{lemma}\label{always}
Let $k\in L^2(\R^d)$ with $\supp k\subset B_{R}(0)$ for some $R>0$\,.
Let moreover $\{v^n\}_{n\in\N}\subset L^2(\R^d)$ be such that $\supp v^n\subset B_R(0)$ and
 $v^n\weakly v$ for some $v\in L^2(\R^d)$\,. Then, $v^n\ast k\to v\ast k$ (strongly) in $L^2(\R^d)$\,.
\end{lemma}
\begin{proof}
We set $g^n:=v^n-v$\,, so that $g^n\weakly 0$ in $L^2(B_R(0))$\,.
 For every $x\in\R^d$\,, $[g^n\ast k](x)\to 0$ as $n\to +\infty$\,. Moreover, let $h\in\R^d$\,; by Young's convolution inequality, we have
 \begin{equation}\label{claim010}
 \Big\|[g^n\ast k](\cdot+h)-[g^n\ast k](\cdot)\Big\|_{L^2(\R^d)}\le \|g^n\|_{L^2(\R^d)}\|k(\cdot+h)-k(\cdot)\|_{L^1(\R^d)}\,;
 \end{equation}
 since the right-hand side of \eqref{claim010} tends to $0$ as $h\to 0$\,, uniformly with respect to $n$\,, by the Fr\'echet-Kolmogorov criterion we get the claim. 
 \end{proof}
\section{$\Gamma$-convergence for the functionals $\J^{\Aa}$ and $\widehat\J^\Aa$ as $\Aa\to 0^+$}
We start by studying the asymptotic behavior of the functionals $-\Aa\J^\Aa$\,.
\begin{theorem}\label{gammafs}
Let $\{\Aa_n\}_{n\in\N}\subset(0,d)$ be such that $\Aa_n\to 0^+$ as $n\to +\infty$. 
\begin{itemize}
\item[(i)]($\Gamma$-liminf inequality) For every $u\in L^2(\Omega)$ and for every $\{u^n\}_{n\in\N}\subset L^2(\Omega)$ with $u^n\weakly u$ in $L^2(\Omega)$, it holds
\begin{equation*}
{d\omega_d}\|u\|^2_{L^2(\Omega)}\le \liminf_{n\to +\infty}-\Aa_n\J^{\Aa_n}(u^n).
\end{equation*}
\item[(ii)]($\Gamma$-limsup inequality) For every $u\in L^2(\Omega)$ 
\begin{equation*}
 \lim_{n\to +\infty}-\Aa_n\J^{\Aa_n}(u)={d\omega_d}\|u\|^2_{L^2(\Omega)}\,.
\end{equation*}
\end{itemize}
\end{theorem}
\begin{proof}
For every $n\in\N$  we set $\bar k^{\Aa_n}:=\Aa_nk^{\Aa_n}$ with $k^\Aa$ defined in \eqref{ks} and, for every $R>\mathrm{diam}(\Omega)$ we set $\bar k^{\Aa_n}_R:=\bar k^{\Aa_n}\chi_{B_R(0)}=\Aa_nk^{\Aa_n}_R$\,.

{\it Proof of (i).} 
Since $\bar k^{\Aa_n}_R\ud z\weakstar {d\omega_d}\delta_0$ we have that $\tilde u\ast \bar k^{\Aa_n}_R\to d\omega_d\tilde u$ (strongly) in $L^2(\Omega)$\,, and hence
\begin{eqnarray}\label{16feb}
\lim_{n\to +\infty}\langle \tilde u,(\tilde u^n-\tilde u)\ast \bar k^{\Aa_n}_R\rangle_{L^2(\R^d)}&=&\lim_{n\to +\infty}\langle \tilde u\ast \bar k^{\Aa_n}_R,\tilde u^n-\tilde u\rangle_{L^2(\R^d)}=0\,,\\ \label{perlimsup}
\lim_{n\to +\infty}\langle \tilde u^n,\tilde u\ast \bar k^{\Aa_n}_R\rangle_{L^2(\R^d)}&=&d\omega_d\|u\|^2_{L^2(\Omega)}\,.
\end{eqnarray}
Moreover, by \eqref{riscr}, we have 
\begin{equation}\label{riscr2}
\begin{aligned}
-\Aa_n\J^{\Aa_n}(u^n)
=&\langle \tilde u^n-\tilde u,(\tilde u^n-\tilde u)\ast \bar k^{\Aa_n}\rangle_{L^2(\R^d)}+\langle \tilde u,(\tilde u^n-\tilde u)\ast \bar k^{\Aa_n}_R\rangle_{L^2(\R^d)}\\
&+\langle \tilde u^n,\tilde u\ast \bar k^{\Aa_n}_R\rangle_{L^2(\R^d)}\,,
\end{aligned}
\end{equation}
and hence, by \eqref{posdef2}, \eqref{16feb}, and \eqref{perlimsup}, we obtain
\begin{equation}\label{anchealtrolim}
\begin{aligned}
\liminf_{n\to +\infty}-\Aa_n\J^{\Aa_n}(u^n)
\ge&\liminf_{n\to +\infty}\langle \tilde u^n-\tilde u,(\tilde u^n-\tilde u)\ast \bar k^{\Aa_n}\rangle_{L^2(\R^d)}\\
&+\liminf_{n\to +\infty}\langle \tilde u,(\tilde u^n-\tilde u)\ast \bar k^{\Aa_n}_R\rangle_{L^2(\R^d)}
+\liminf_{n\to +\infty}\langle \tilde u^n,\tilde u\ast \bar k^{\Aa_n}_R\rangle_{L^2(\R^d)}\\
\ge&d\omega_d\|u\|^2_{L^2(\Omega)}\,.
\end{aligned}
\end{equation}

{\it Proof of (ii).} The claim follows directly by \eqref{perlimsup} and by the first equality in \eqref{riscr}.
\end{proof}
\begin{remark}\label{secambiosegno}
\rm{
Let $\{\Aa_n\}_{n\in\N}\subset(0,d)$ with $\Aa_n\to 0^+$ as $n\to +\infty$\,.
By \eqref{16feb}-\eqref{riscr2}, for any $u\in L^2(\Omega)$ and for any sequence $\{u^n\}_{n\in\N}\subset L^2(\Omega)$ with $u^n\to u$ (strongly) in $L^2(\Omega)$\,, we have that
\begin{equation*}
\lim_{n\to +\infty}\Aa_n\J^{\Aa_n}(u^n)=-d\omega_d\|u\|^2_{L^2(\Omega)}\,,
\end{equation*}
which shows, in particular, that the $\Gamma$-limit of the functionals $\Aa\J^{\Aa}$ (as $\Aa\to 0^+$) with respect to the strong convergence in $L^2(\Omega)$ is the functional $-d\omega_d\|\cdot\|^2_{L^2(\Omega)}$\,. 
This fact implies that the $\Gamma$-limit of  $\Aa\J^{\Aa}$ (as $\Aa\to 0^+$) with respect to the weak convergence in $L^2(\Omega)$ equals to $-\infty$\,.
}
\end{remark}
In Theorem \ref{thm:stozerofo} below we show that the functional $\widehat\J^{\Aa}$ $\Gamma$-converges to the functional $\widehat\J^{0}$\,.
To this end, following \cite[Remark 1.3]{CDKNP}, we introduce the space 
\begin{equation*}
H^0_0(\Omega):=\{u\in L^2(\Omega)\,:\,\G^0_1(u)<+\infty\}\,.
\end{equation*}
Such a functional space could be endowed with a $0$-Gagliardo type norm defined as
\begin{equation*}
\|u\|_{H_0^0(\Omega)}:=\|u\|_{L^2(\Omega)}+\big(2\G^0_1(u)\big)^{\frac 1 2}\,.
\end{equation*}
\begin{theorem}\label{thm:stozerofo}
Let $\{\Aa_n\}_{n\in\N}\subset(0,d)$ be such that $\Aa_n\to 0^+$ as $n\to +\infty$\,. The following $\Gamma$-convergence result holds true.
\begin{itemize}
\item[(i)] (Compactness) Let $\{u^n\}_{n\in\N}\subset L^2(\Omega)$ be such that
\begin{equation}\label{COMP2}
\widehat\J^{\Aa_n}(u^n)+ (|\Omega|+m) \| u^n \|^2_{L^2(\Omega)}\le M,
\end{equation} 
for some constants $m,M>0$ independent of $n$\,. Then, up to a subsequence, $u^n\to u$ strongly in $L^2(\Omega)$ for some $u\in H_0^0(\Omega)$\,.
\item[(ii)] ($\Gamma$-liminf inequality) For every $u\in H^0_0(\Omega)$ and for every $\{u^n\}_{n\in\N}\subset L^2(\Omega)$ with $u^n\to u$ in $L^2(\Omega)$\,, it holds
$$
\widehat\J^0(u)\le\liminf_{n\to +\infty}\widehat\J^{\Aa_n}(u^n)\,.
$$
\item[(iii)] ($\Gamma$-limsup inequality) For every $u\in H_0^0(\Omega)$ there exists $\{u^n\}_{n\in\N}\subset L^2(\Omega)$ with $u^n\to u$ in $L^2(\Omega)$ such that
$$
\widehat\J^0(u)=\lim_{n\to +\infty}\widehat\J^{\Aa_n}(u^n)\,.
$$
\end{itemize}
\end{theorem}
In order to prove (i) of Theorem \ref{thm:stozerofo} we adopt a strategy similar to that used in the proof of \cite[Theorem 1.1]{JW} . To this end we prove the following result.
\begin{lemma}\label{convlemma}
Let $\{\Aa_n\}_{n\in\N}\subset(0,d)$ be such that $\Aa_n\to 0^+$ as $n\to +\infty$\,. 
Let moreover $\{u^n\}_{n\in\N}\subset L^2(\Omega)$ 
with
\begin{equation}\label{boundcomp}
\|u^n\|_{L^2(\Omega)}+\G^{\Aa_n}_1(u^n)\le C\,,
\end{equation}
for some constant $C>0$ (independent of $n$)\,. Then, up to a subsequence, $u^n\to u$ in $L^2(\Omega)$ for some $u\in H^0_0(\Omega)$\,.
\end{lemma}
\begin{proof}
By \eqref{boundcomp},  there exists $u\in L^2(\Omega)$ such that, up to a (not-relabeled) subsequence, $u^n\weakly u$ in $L^2(\Omega)$\,. 

For every $\Aa\in[0,d)$
and for every $0<\delta<1$\,, we define $j_\delta^{\Aa}:\R^d\to (0,+\infty)$ as
$$
j_\delta^{\Aa}(z):=
\frac{\chi_{A_{\delta,1}(0)}(z)}{|z|^{d-\Aa}}\,,
$$
where we have set
 $A_{r,R}(\xi):=B_{R}(\xi)\setminus \overline{B}_{r}(\xi)$ for every $0<r<R$ and for every $\xi\in\R^d$\,.  
Furthermore, for every $\Aa\in[0,d)$ and for every $0<\delta<1$ we set  $w_\delta^{\Aa}:=\frac{j_\delta^{\Aa}}{\|j_\delta^{\Aa}\|_{L^1(\R^d)}}$\,.

Notice that
\begin{equation}\label{diretto}
\|j_{\delta}^{\Aa_n}\|_{L^1(\R^d)}=d\omega_d\frac{1-\delta^{\Aa_n}}{\Aa_n}\quad\textrm{and}\quad
\sup_{n\in\N}\|j_\delta^{\Aa_n}\|_{L^2(\R^d)}\le\sup_{n\in\N} \frac{\omega_d^{\frac 1 2}}{\delta^{d-\Aa_n}}=\frac{\omega_d^{\frac 1 2}}{\delta^d}\,.
\end{equation}

We observe that for every $0<\delta<1$ 
\begin{equation}\label{convej}
j^{\Aa_n}_\delta\to j^{0}_\delta\textrm{ in }C^k(\overline{A_{\delta,1}(0)}) \text{ for all } k\in\N  \textrm{ and } \|j^{\Aa_n}_\delta\|_{L^1(\R^d)}\to d\omega_d|\log\delta|= \|j^{0}_\delta\|_{L^1(\R^d)}\,,
\end{equation}
and, by Lemma \ref{always},
\begin{equation}\label{claim01}
\|w^0_\delta\ast (u^n-u)\|_{L^2(\R^d)}\to 0\qquad\textrm{as }n\to +\infty.
\end{equation}
In turn, by  triangular and convolution's Young inequalities, using \eqref{convej} and \eqref{claim01}, we get
\begin{multline}\label{claim1}
\|w^{\Aa_n}_\delta\ast u^n- w^0_\delta\ast u\|_{L^2(\R^d)}\\
\le \|w^{\alpha_n}_\delta-w^0_\delta\|_{L^1(\R^d)}\|u^n\|_{L^2(\R^d)}+\|w^0_\delta\ast (u^n-u)\|_{L^2(\R^d)}\to 0\qquad\textrm{ as }n\to +\infty\,.
\end{multline}
Now, following \cite[Lemma 2.2]{JW}, using Jensen's inequality,
   we get
   \begin{equation}\label{claim2}
   \begin{aligned}
\|u^n-w_\delta^{\Aa_n}\ast u^n\|^2_{L^2(\R^d)}=&\int_{\R^d}\ud x\bigg(\int_{\R^d}(u^n(x+z)-u^n(x))w_\delta^{\alpha_n}(z)\ud z\bigg)^2\\
\le&\int_{\R^d}\ud x\int_{\R^d}|u^n(x+z)-u^n(x)|^2w_\delta^{\alpha_n}(z)\ud z\\
\le&\frac{2}{\|j^{\alpha_n}_\delta\|_{L^1(\R^d)}}\G^{\Aa_n}_1(u^n) \le \frac{C}{|\log\delta|}\,,
\end{aligned}
\end{equation}
where in the last inequality we have used \eqref{boundcomp} and \eqref{convej}. 
Moreover, since $w^{0}_{\delta}\ud z\weakstar \delta_0$ as $\delta\to 0$\,, we have that
\begin{equation}\label{claim3}
\|w^0_\delta\ast u-u\|_{L^2(\R^d)}\to 0\qquad\textrm{as }\delta\to 0\,.
\end{equation}
Therefore, by triangular inequality, we obtain that, for every $0<\delta<1$\,,
\begin{equation*}
\begin{aligned}
\|u^n-u\|_{L^2(\R^d)}\le \|u^n-w_\delta^{\Aa_n}\ast u^n\|_{L^2(\R^d)}+\|w_\delta^{\Aa_n}\ast u^n-w_\delta^0\ast u\|_{L^2(\R^d)}+\|w^0_\delta\ast u-u\|_{L^2(\R^d)}\,,
\end{aligned}
\end{equation*}
whence the strong convergence of $u^n$ to $u$ follows by and by \eqref{claim1}, \eqref{claim2}, and \eqref{claim3}, sending first $n\to +\infty$ and then $\delta\to 0$\,.
Finally, by Fatou Lemma and by \eqref{boundcomp},
\begin{equation}\label{ancheinliminf}
\G_1^0(u)\le \liminf_{n\to +\infty}\G_1^{\Aa_n}(u_n)\le C\,,
\end{equation}
whence we deduce that $u\in H^0_0(\Omega)$\,, thus concluding the proof of the lemma.
\end{proof}
\begin{proof}[Proof of Theorem \ref{thm:stozerofo}]
We start by proving compactness.
By \eqref{COMP2}, \eqref{Jshat_re} and \eqref{estJs1}, we have
\begin{equation*}
M\ge
\G^{\Aa_n}_1(u^n)+m\|u^n\|^2_{L^2(\Omega)}\,,
\end{equation*}
which, in view of Lemma \ref{convlemma},  yields (i).
Moreover, in view of \eqref{Jshat_re}, the proof of (ii) follows by \eqref{ancheinliminf} and by \eqref{DCT2}.
Finally, the proof of (iii) follows by using standard density arguments in $\Gamma$-convergence (see for instance the proof of \cite[Theorem 1.4 (iii)]{CDKNP}).
\end{proof}
\begin{remark}\label{secambiosegno2}
\rm{
Analogously to what already observed in Remark \ref{secambiosegno}, one can easily prove that the functionals $-\widehat\J^{\alpha}$ $\Gamma$-converge to $-\infty$ (as $\alpha\to 0^+$) with respect to the strong $L^2$ convergence. For example, assume that $0\in\Omega$ and consider $v_n:=\frac{n^{\frac d 2}}{\log^{\frac 1 4}n}\chi_{B_{\frac 1 n}(0)}$\,. Then, one can easily check that $v_n\to 0$ strongly in $L^2(\Omega)$\,. Moreover, for $n$ large enough,
$\G^{\alpha}_1(v_n)\ge C\frac{1-n^{-\alpha }}{\alpha}\frac{1}{\log^{\frac{1}{2}}n}$\,; since  
$$
\lim_{n\to +\infty}\lim_{\alpha\to 0}\frac{1-n^{-\alpha }}{\alpha}\frac{1}{\log^{\frac{1}{2}}n}=\lim_{n\to +\infty}\log^{\frac 1 2} n=+\infty\,,
$$
by a standard diagonal argument there exists a sequence $\{v_\alpha\}_\alpha=\{v_{n(\alpha)}\}$ with  $\|v_\alpha\|_{L^2(\Omega)}\to 0$ and such that $\G^{\alpha}_1(v_{\alpha})\to +\infty$ (as $\alpha\to 0^+$). As a consequence, for every $u\in L^2(\Omega)$\,, $u+v_{\alpha}\to u$ (strongly) in $L^2(\Omega)$ and $\G^{\alpha}_1(u+v_{\alpha})\to +\infty$\,. 
Recalling \eqref{Jshat_re} and recalling Remark \ref{DCT}, we conclude that $-\widehat\J^{\alpha}(u+v_\alpha)\to -\infty$ for every $u\in L^2(\Omega)$\,.
}
\end{remark}
\section{$\Gamma$-convergence for the functionals $\J^\Aa$ as $\Aa\to d^-$}
We define the functional $\J^{d}:L^2(\Omega)\to(-\infty,0]$ as $\J^{d}(u):=-\Big(\int_{\Omega}u(x)\ud x\Big)^2$\,.
\begin{theorem}\label{thm:sto-d2}
Let $\{\Aa_n\}_{n\in\N}\subset(0,d)$ be such that $\Aa_n\to {d}^-$ as $n\to +\infty$\,.
For every $u\in L^2(\Omega)$ and for every $\{u^n\}_{n\in\N}\subset L^2(\Omega)$ with $u^n\weakly u$ in $L^2(\Omega)$ it holds
\begin{equation}\label{nuovatesi}
\lim_{n\to +\infty} \J^{\Aa_n}(u^n)=\J^{d}(u)
\,.
\end{equation}
In particular, as $\Aa\to d^-$\,, the functionals $\pm\J^{\Aa}$  $\Gamma$-converge to the functional $\pm\J^{d}$ with respect to the weak $L^2$ convergence in $L^2(\Omega)$\,.


\end{theorem}
\begin{proof}
We first notice that for any compact set $K\subset\R^d$
\begin{equation}\label{convecar}
\lim_{n\to +\infty}\|k^{\Aa_n}- 1\|_{L^1(K)}=0\,,
\end{equation}
where $k^{\Aa_n}$ is defined in \eqref{ks}\,.
By \eqref{riscr}, we thus have that
\begin{equation}\label{spe00}
\J^{\Aa_n}(u^n)=-\langle \tilde u^n, \tilde u^n\ast(k^{\Aa_n}-1)\rangle_{L^2(\R^d)}-\langle \tilde u^n, \tilde u^n\ast 1\rangle_{L^2(\R^d)}\,.
\end{equation}
Now, taking $R>\mathrm{diam}(\Omega)$\,, by Young's convolution inequality and by \eqref{convecar}, we have
\begin{equation}\label{spe02}
\limsup_{n\to +\infty}\langle \tilde u^n, \tilde u^n\ast(k^{\Aa_n}-1)\rangle_{L^2(\R^d)}\le\limsup_{n\to +\infty}\|u^n\|^2_{L^2(\Omega)}\| k^{\Aa_n}- 1\|_{L^1(B_R(0))}=0\,,
\end{equation}
and, since $u^n\weakly u$ in $L^2(\Omega)$\,, 
\begin{equation}\label{spe03}
\lim_{n\to +\infty}\langle \tilde u^n, \tilde u^n\ast 1\rangle_{L^2(\R^d)}= \langle\tilde u,\tilde u\ast 1\rangle_{L^2(\R^d)}=-\J^{d}(u)
\,.
\end{equation}
By \eqref{spe00}, \eqref{spe02} and \eqref{spe03}, we obtain \eqref{nuovatesi}.


\end{proof}
For every $\Aa\in (0,d)$  
we define the functional $\widetilde{\J}^\Aa:L^2(\Omega)\to \R$ as
\begin{equation*}
\widetilde{\J}^\Aa(u):=\frac{\J^\Aa(u)-\J^{d}(u)}{d-\Aa}\,;
\end{equation*}
moreover, 
and for every $u\in L^2(\Omega)$ we define
\begin{equation}\label{fo-d2}
\widetilde\J^{d}(u):=-\int_{\Omega}\int_{\Omega}u(x)u(y)\tilde k^{d}(|x-y|)\ud y\ud x\,,
\end{equation}
where we have set
$\tilde k^{d}(z):=\log\frac{1}{|z|}$\,. 
Notice that 
\begin{equation}\label{ovvia}
\tilde k^{d}\in L^p_{\mathrm{loc}}(\R^d)\qquad\textrm{for any }p\in[1,+\infty)\,;
\end{equation}
therefore, for every $R>\mathrm{diam}(\Omega)$\,, by Young's convolution inequality, we get 
\begin{equation}\label{lampos}
|\widetilde\J^{d}(u)|\le \|\tilde k^{d}\|_{L^1(B_R(0))}\|u\|^2_{L^2(\Omega)}\qquad\textrm{for every }u\in L^2(\Omega)\,,
\end{equation}
which shows, in particular, that the functional $\widetilde\J^{d}$ defined by \eqref{fo-d2} takes values in $\R$\,.
\begin{remark}\label{quadrpos}
\rm{
Let $k\in L^1_{\mathrm{loc}}(\R^d)$ and let $\mathscr{J}_k$ be the functional defined on $L^2(\Omega)$  by
\begin{equation*}
\mathscr{J}_k(u):=\int_{\Omega}\int_{\Omega}u(x)u(y)k(x-y)\ud y\ud x\qquad\textrm{for every }u\in L^2(\Omega)\,.
\end{equation*} 
Let $R\ge\mathrm{diam}(\Omega)$\,.
Then, by H\"older and Young's convolution inequalities,
\begin{equation*}
|\mathscr{J}_k(u)|\le \|u\|_{L^2(\Omega)}^2\|k\|_{L^1(B_R(0))}\,.
\end{equation*}

}
\end{remark}  
\begin{remark}\label{orasi}
\rm{
Let $\{k^n\}_{n\in\N}\subset L^1_{\mathrm{loc}}(\R^d)$ with $\sup_{n\in\N}\|k^n\|_{L^1(B_R(0))}\le C$\,, for some $C>0$ and $R>\mathrm{diam}(\Omega)$\,. Then, by Remark \ref{quadrpos} the functionals $\mathscr{J}_{k^n}$ are $\lambda$-positive,  for any $\lambda>2C$\,. Therefore, for any sequence $\{u_n\}_{n\in\N}\subset L^2(\Omega)$ with 
$$
\sup_{n\in\N}\mathscr{J}_{k^n}(u^n)+\lambda \|u^n\|^2_{L^2(\Omega)}<+\infty\,,
$$
we have that, up to a subsequence, $u^n\weakly u$ in $L^2(\Omega)$\,, for some $u\in L^2(\Omega)$\,. 

In particular,  this applies to any sequence of functionals $\{\J^{\alpha_n}\}_{n}$ with $\alpha_n\to d^-$\,.
}
\end{remark}
\begin{theorem}\label{thm:sto-d2fo}
Let $\{\Aa_n\}_{n\in\N}\subset(0,d)$ be such that $\Aa_n\to d^-$ as $n\to +\infty$\,.

For every $u\in L^2(\Omega)$ and for every $\{u^n\}_{n\in\N}\subset L^2(\Omega)$ with $u^n\weakly u$ in $L^2(\Omega)$ it holds
\begin{equation}\label{nuovatesifo}
\lim_{n\to +\infty} \widetilde\J^{\Aa_n}(u^n)=\widetilde\J^{d}(u)
\,.
\end{equation}
In particular, as $\Aa\to d^-$\,, the functionals $\pm\J^{\Aa}$  $\Gamma$-converge to the functional $\pm\J^{d}$ with respect to the weak $L^2$ convergence in $L^2(\Omega)$\,.

\end{theorem}
\begin{proof}
For any $n\in\N$ we set $\tilde k^{\Aa_n}:=\frac{k^{\Aa_n}-1}{d-\Aa_n}$\,, with $k^{\Aa_n}$ defined in \eqref{ks}.
For any $v\in L^2(\Omega)$\,, it holds
\begin{equation}\label{riscr4}
\widetilde\J^{\Aa_n}(v)=-\langle \tilde v,\tilde v\ast \tilde k^{\Aa_n}\rangle_{L^2(\R^d)}\,,\qquad
\widetilde\J^{d}(v)=-\langle \tilde v,\tilde v\ast \tilde k^{d}\rangle_{L^2(\R^d)}\,,
\end{equation}
and hence
\begin{equation}\label{lausodopo}
\widetilde\J^{\Aa_n}(v)=-\langle \tilde v,\tilde v\ast (\tilde k^{\Aa_n}-\tilde k^{d}) \rangle_{L^2(\R^d)}
+\widetilde\J^{d}(v)\\
\end{equation}
Moreover, by Taylor expansion, for any $z\in\R^d\setminus\{0\}$\,, we have 
\begin{equation}\label{tay}
\tilde k^{\Aa_n}(z)=\tilde k^{d}(z)+\mathrm{O}\Big((d-\Aa_n)\log^2\frac{1}{|z|}\Big)\,,
\end{equation}
with $\lim_{t\to 0}\frac{\mathrm{O}(t)}{t}\in\R$\,. Therefore, for any compact set $K\subset\R^d$\,, it holds
\begin{equation}\label{convespe}
\lim_{n\to +\infty}\|\tilde k^{\Aa_n}-\tilde k^{d}\|_{L^p(K)}= 0\qquad\textrm{for any }p\in[1,+\infty)\,.
\end{equation}
Let $R>\mathrm{diam}(\Omega)$\,. 

We start by proving (a).
By Young's convolution inequality and by \eqref{convespe}, we have
\begin{equation}\label{spe2}
\limsup_{n\to +\infty}\langle \tilde u^n, \tilde u^n\ast(\tilde k^{\Aa_n}-\tilde k^{d})\rangle_{L^2(\R^d)}\le\limsup_{n\to +\infty}\|u^n\|^2_{L^2(\Omega)}\|\tilde k^{\Aa_n}-\tilde k^{d}\|_{L^1(B_R(0))}=0\,,
\end{equation}
and, by Lemma \ref{always}, 
\begin{equation}\label{spe3}
\lim_{n\to +\infty}\widetilde\J^d( u^n)=\widetilde\J^d(u)\,.
\end{equation}
By \eqref{lausodopo}, \eqref{spe2}, and \eqref{spe3}, we thus obtain
\begin{equation*}
\lim_{n\to +\infty}\widetilde\J^{\Aa_n}(u^n)=\widetilde\J^d(u)\,,
\end{equation*}
i.e., \eqref{nuovatesifo}.

\end{proof}
\begin{remark}\label{comptilded}
\rm{
By Remark \ref{orasi} and by \eqref{lausodopo}, \eqref{convespe} and \eqref{lampos},  it immediately follows that any sequence $\{\widetilde\J^{\alpha_n}\}_{n}$ with $\alpha_n\to d^{-}$ is uniformly $\lambda$-positive and hence  weakly equi-coercive in $L^2(\Omega)$ (in the sense of Remark \ref{orasi}).
}
\end{remark}

\section{Convergence of the $\Aa$-Riesz flows}
\subsection{The $\Aa$-Riesz potential for $\Aa\in (0,d)$ and for $\Aa=0$ and $\Aa=d$}
For every $\Aa\in(0,d)$ and for every $v\in L^2(\R^d)$ the $\Aa$-Riesz potential is formally defined by 
\begin{equation*}
(-\Delta)^{-\frac{\Aa}{2}}v(x):=2\int_{\R^d}\frac{v(y)}{|x-y|^{d-\Aa}}\ud y\,,\qquad\qquad\textrm{for a.e. } x\in\R^d\,.
\end{equation*}
Notice that such a definition coincides, up to  multiplicative constants depending on $\alpha$ and $d$\,, with  \cite[formula (4), p. 117]{Stein}.
By \cite[Theorem 1, p. 119]{Stein} the above definition is well posed, and there exists a constant $C(d,\Aa)$ such that
\begin{equation*}
\|(-\Delta)^{-\frac{\Aa}{2}} v\|_{L^q(\R^d)}\le C(d,\Aa)\|v\|_{L^2(\R^d)}\qquad\qquad\textrm{for all }v\in L^2(\R^d)\,,
\end{equation*}
where $q=\frac{2d}{d-\Aa}$. Since $q\ge 2$, by H\"older inequality, we deduce that 
\begin{equation}\label{lapl2}
\|(-\Delta)^{-\frac \Aa 2} v\|_{L^2(\Omega)}\le C(d,\Aa,\Omega)\|v\|_{L^2(\R^d)};
\end{equation} we also notice that, if $v\in L^2(\Omega)$ with $v\equiv 0 $ on $\R^d\setminus \Omega$, then  \eqref{lapl2}  is an easy consequence of convolution's Young inequality. 
With a little abuse of notation, for every $\Aa\in (0,d)$ and for every
 $u\in L^2(\Omega)$ we set $(-\Delta)^{-\frac \Aa 2}u:=(-\Delta)^{-\frac \Aa 2}\tilde u$\,.

One can easily check that the $\Aa$-Riesz potential is nothing but the first variation of the $\Aa$-Riesz functionals $-\J^{\Aa}$, i.e., for every $u,\ffi\in L^2(\Omega)$ we have
\begin{equation}\label{eq:firstvars}
\lim_{t\to 0}\frac{\J^\Aa(u+t\ffi)-\J^\Aa(u)}{t}=-\langle(-\Delta)^{-\frac \Aa 2} u,\ffi\rangle_{L^2(\Omega)}.
\end{equation}
In analogy with \cite[Subsection 4.1]{CDKNP},
for every $\psi\in\CC^{\infty}_{\cc}(\R^d)$ we define the $0$-fractional laplacian of $\psi$ as
\begin{equation*}
(-\Delta)^0\psi(x):=\int_{B_1}\frac{2\psi(x)-\psi(x+z)-\psi(x-z)}{|z|^d}\ud z-2\int_{\R^d\setminus \overline{B}_1}\frac{\psi(x+z)}{|z|^d}\ud z\,,\quad x\in\R^d \, .
\end{equation*}
Furthermore, for every $u\in H^0_0(\Omega)$
we define $0$-fractional laplacian of $u$ by duality as 
\begin{equation}\label{ext0lap}
\langle (-\Delta)^0u,\ffi\rangle :=\langle u,(-\Delta)^0\tilde\ffi\rangle_{L^2(\Omega)}\,,\qquad\textrm{for all }\ffi\in\CC^{\infty}_{\cc}(\Omega)\,.
\end{equation}
Clearly, the $0$-fractional laplacian is the first variation of the functional $\widehat \J^0$\,, as shown in \cite[Proposition 3.2]{CDKNP}.
      
Finally, for every $v\in L^2(\Omega)$ we set 
\begin{equation*}
(-\Delta)^{-\frac d 2}v(x):=2\int_{\Omega}v(y)\log\frac{1}{|x-y|}\ud y\,,\qquad x\in\R^d\,,
\end{equation*}
and we notice that $(-\Delta)^{-\frac d 2}: L^2(\Omega)\to \mathrm{C}(\R^d)$\,.

It is easy to see that $(-\Delta)^{-\frac d 2}$ is the first variation of the functional $-\widetilde\J^{d}$ defined in \eqref{fo-d2}\,, i.e., that for any $u,\ffi\in L^2(\Omega)$
\begin{equation*}
\lim_{t\to 0}\frac{\widetilde\J^{d}(u+t\ffi)-\widetilde\J^{d}(u)}{t}=-\langle(-\Delta)^{-\frac d 2}u,\ffi\rangle_{L^2(\Omega)}.
\end{equation*}

\subsection{Abstract stability results for parabolic flows}\label{astrrisu}
It is well known that, under suitable uniform convexity assumptions, $\Gamma$-convergence commutes with gradient flows. This fact, first exploited in \cite{Bre}, has been generalized in many respects, such as for instance in metric spaces \cite{DeGMT, AGS, SS}. For  the specific framework of  uniform $\lambda$-convex functionals we refer for instance to \cite{NSV, O1, O2, CDKNP}. Here, we recall the stability result as appeared in \cite[Theorem 3.8]{CDKNP},
that best suits our purposes. 
We start by recalling the notion of $\lambda$-positive and $\lambda$-convex functions.
Let $\h$ be a Hilbert space and let $\lambda>0$\,. We say that a function $\F: \h\rightarrow (-\infty, +\infty]$ is $\lambda$-convex if the function $\F(\cdot)+ \frac{\lambda}{2} \vert \cdot \vert_\h^2$ is convex.
Moreover, we say that $\F$ is $\lambda$-positive if 
$\F(x)+ \frac{\lambda}{2} \vert x \vert_\h^2 \geq 0$ for every $x \in\h$\,.

\begin{theorem}\label{genstab}
Let $\h$ be a Hilbert space. Let $\{\F^n\}_{n\in\N}$ with $\F^n:\h\to(-\infty,+\infty]$ for every $n\in\N$ be a sequence of proper, strongly lower semicontinuous functions which are $\lambda$-convex and $\lambda$-positive, for some $\lambda>0$ independent of $n$\,. Let $\{x_0^n\}_{n\in\N}\subset\h$ be such that $\sup_{n\in\N}\F^n(x_0^n)<+\infty$ and $x_0^n\to x^\infty_0$
for some $x^\infty_0\in\h$\,. Assume that one of the following statements is satisfied:
\begin{itemize}
\item[(a)] 
The functions $\F^n$ $\Gamma$-converge to some proper function $\F^\infty$ with respect to the weak $\h$-convergence. 
 Moreover, 
 the $\Gamma$-limsup inequality is satisfied with respect to the strong $\h$-convergence, i.e., for every $y\in\h$ there exists a sequence $\{y^n\}_{n\in\N}$ with $y^n\overset{\h}{\to} y$ such that $\F^n(y^n)\to \F^{\infty}(y)$ as $n\to +\infty$\,.
  
\item[(b)] The functions $\F^n$ $\Gamma$-converge to some proper function $\F^\infty$ with respect to the strong $\h$-convergence (as $n\to +\infty$) and every sequence $\{y^n\}_{n\in\N}\subset\h$ with
\[
\sup_{n\in\N}\F^n(y^n)+\frac{\lambda}{2}|y^n|_\h^2<+\infty
\]
admits a strongly convergent subsequence.
\end{itemize}
Let $T>0$\,.
Then, $\F^{\infty}(x_0^\infty) <+\infty$ and the unique solutions $x^n \in H^1([0,T];\H)$ of the Cauchy problem
\begin{equation}\label{caun}
\begin{cases}
 \dot{x}(t)\in -\partial \F^n(x(t))  \qquad \text{for a.e.\ }t \in (0,T)\,, \\
x(0)=x_0^n
\end{cases}
\end{equation}  
weakly converge, as $n\to+\infty$\,, in $H^1([0,T];\h)$ to the unique solution $x^\infty$ of the problem
\begin{equation}\label{cauninf}
\begin{cases}
 \dot{x}(t)\in -\partial \F^\infty(x(t)) \qquad \text{for a.e.\ }t \in (0,T)\,, \\
x(0)=x^\infty_0\,.
\end{cases}
\end{equation}  
Furthermore, if 
\begin{equation}\label{reco0}
\lim_{n\to +\infty} \F^n(x^n_0)=\F^\infty(x^\infty_0)\,,
\end{equation} 
then,  we have that 
\begin{equation}\label{strongc}
x^n\to x^\infty\qquad\textrm{(strongly) in }H^1([0,T];\h)\,,
\end{equation}
\begin{equation}\label{alwaysrs}
x^n(t)\overset{\h}{\to}x^{\infty}(t)\quad\textrm{and}\quad\F^n(x^n(t))\to\F^{\infty}(x^{\infty}(t))\qquad\textrm{for every }t\in[0,T]\,. 
\end{equation}
\end{theorem}


\subsection{Convergence of the $\Aa$-parabolic flows as $\Aa\to 0^+$}
Here we state and prove the  convergence results for the parabolic flows corresponding to the (either scaled or renormalized) $\Aa$-Riesz functionals as $\Aa\to 0^+$\,.  These results follow by applying our $\Gamma$-convergence analysis together with  the stability result Theorem \ref{genstab}.
 
\begin{theorem}\label{convheat0ordsto0}
Let $\{\Aa_n\}_{n\in\N}\subset (0,d)$ be such that $\Aa_n\to 0^+$ as $n\to +\infty$\,. Let $u^\infty_0\in L^2(\Omega)$ and $\{u_0^n\}_{n\in\N}\subset L^2(\Omega)$ be such that 
 $u^n_0\to u^\infty_0$ in $L^2(\Omega)$\,.
Let $T>0$; then,
for every $n\in\N$ there exists a unique  solution $u^n\in H^1([0,T]; L^2(\Omega))$ of
\begin{equation}\label{cauchyord0n}
\begin{dcases}
 u_t(t) =-\Aa_n(-\Delta)^{-\frac{\Aa_n}{2}}u(t) \qquad\textrm{for a.e.\ }t\in(0,T)\\
u(0)=u^{n}_0\,.
\end{dcases}
\end{equation}
Moreover,  $u^n\to u^\infty$ in $H^1([0,T];L^2(\Omega))$ as $n\to +\infty$\,, where $u^\infty$ 
 is the unique  solution of
\begin{equation}\label{cauchyord0infty}
\begin{dcases}
 u_t(t) =-2d\omega_d u(t)\qquad\textrm{for a.e.\ }t\in (0,T)\,,\\
u(0)=u^\infty_0\,,
\end{dcases}
\end{equation}
and
$$
\|u^n(t)-u^\infty(t)\|_{L^2(\Omega)}\to 0\quad\textrm{and}\quad -\Aa_n\J^{\Aa_n}(u^n(t))\to d\omega_d\|u^\infty(t)\|^2_{L^2(\Omega)}\qquad\textrm{for every }t\in [0,T] \,.
$$
\end{theorem}
\begin{proof}
Recall that,
by \eqref{welldefJs}, the functionals $-\Aa_n\J^{\Aa_n}$ are finite on $L^2(\Omega)$ and, by \eqref{posdef2} and Remark \ref{quadrpos}, they are also positive and convex. 
Furthermore, by \eqref{continJs}, the functionals $-\Aa_n\J^{\Aa_n}$ are continuous with respect to the strong convergence in $L^2(\Omega)$\,. 
Moreover, for every $u,\ffi\in L^2(\Omega)$\,,
 \begin{equation}\label{banale}
 \lim_{t\to 0}\frac{d\omega_d\|u+t\ffi\|^2_{L^2(\Omega)}-d\omega_d\|u\|^2_{L^2(\Omega)}}{t}=2\langle d\omega_d u,\ffi\rangle_{L^2(\Omega)}\,.
 \end{equation}
Finally, by Remark \ref{secambiosegno}, $\lim_{n\to +\infty}-\Aa_n\J^{\Aa_n}(u^n_0)=d\omega_d\|u^\infty_0\|^2_{L^2(\Omega)}$\,.
 Hence, the conclusion follows by applying Theorem \ref{genstab} with $\H=L^2(\Omega)$\,, $\F^n = -\Aa_n\J^{\Aa_n}$ and $\F^{\infty}(\cdot)=d\omega_d\|\cdot\|^2_{L^2(\Omega)}$\,, once noticed that, in view of Theorem \ref{gammafs}, assumption (a)  is satisfied. 
\end{proof}
\begin{remark}\label{anchequestovale}
\rm{
For every $n\in\N$ let $u^n$ be the unique solution to the Cauchy problem \eqref{cauchyord0n} and let $u$ be the solution to \eqref{cauchyord0infty} with $-2d\omega_d$ replaced by $2d\omega_d$\,. Existence and uniqueness of the solutions for the problems above are consequence of the classical gradient flow theory for $\lambda$-convex functionals (see, for instance, \cite{AGS, NSV} or \cite[Theorem 3.2]{CDKNP}). It is a natural question if $u^n$ converge to $u$\,. In the geometric context of characteristic functions this is the case for the corresponding curvature flows. In the present setting, unfortunately, Theorem \ref{genstab} does not apply. In fact, on the one hand, 
compactness can be achieved only in the weak $L^2$ topology, so that assumption (b) does not hold true.
On the other hand, in view of Remark \ref{secambiosegno}, the $\Gamma$-limit of the functionals $\alpha\J^{\alpha}$ (as $\alpha\to 0^+$) is $-d\omega_d\|\cdot\|^2_{L^2}$ only with respect to the strong $L^2$ topology, so that neither assumption (a) is satisfied.
More in general, it would be interesting
 to detect natural assumptions weaker than (a) and (b)  that guarantee convergence for similar parabolic flows, 
 when the $\Gamma$-convergence property  holds true with respect to a topology that is  stronger than that for which equicoercivity can be provided.
}
\end{remark}
\begin{theorem}\label{convheat1ordsto0}
Let $\{\Aa_n\}_{n\in\N}\subset (0,d)$ be such that $\Aa_n\to 0^+$ as $n\to +\infty$\,. Let $u^\infty_0\in L^2(\Omega)$ and let $\{u^{n}_0\}_{n\in\N}\subset L^2(\Omega)$ be such that $sup_{n\in\N}\widehat{\J}^{\Aa_n}(u^n_0)<+\infty$
and
 $u^n_0\to u^\infty_0$ in $L^2(\Omega)$\,. Let $T>0$; then, 
for every $n\in\N$ there exists a unique  solution $u^n\in H^1([0,T]; L^2(\Omega))$ of
\begin{equation}\label{cauchyord1n}
\begin{dcases}
 u_t(t) =- \Big[-(-\Delta)^{-\frac{\Aa_n}{2}}u(t)+2\frac{d\omega_d}{\Aa_n}u(t) \Big] \qquad\textrm{for a.e.\ }t\in (0,T)\,,\\
u(0)=u^{n}_0\,.
\end{dcases}
\end{equation}
Moreover, $u^\infty_0\in H^0_0(\Omega)$ and   $u^n\weakly u^\infty$ in $H^1([0,T];L^2(\Omega))$ as $n\to +\infty$\,, where 
$u^\infty$ is the unique solution of
\begin{equation}\label{cauchyord1infty}
\begin{dcases}
 u_t(t) =-(-\Delta)^0 u(t)\qquad\textrm{for a.e.\ }t\in (0,T)\\
u(0)=u^\infty_0\,.
\end{dcases}
\end{equation}
Furthermore, if
$$
\lim_{n\to +\infty} \widehat{\J}^{\Aa_n}(u^n_0)=\widehat \J^0(u^\infty_0)\,,
$$
then,  $u^n\to u^\infty$ (strongly) in $H^1([0,T];L^2(\Omega))$ and
$$
\|u^n(t)-u^\infty(t)\|_{L^2(\Omega)}\to 0\quad\textrm{and}\quad \widehat{\J}^{\Aa_n}(u^n(t))\to \widehat{\J}^0(u^\infty(t))\qquad\textrm{for every } t\in [0,T]\,.
$$
\end{theorem}
\begin{proof}
By \eqref{Jshat_re} and \eqref{estJs1}, the functionals 
$\widehat\J^{\Aa_n}$ are $\lambda$-convex and $\lambda$-positive for every $\lambda>2|\Omega|$\,. Moreover, by Remark \ref{DCT}, they are lower semicontinuous with respect to the strong convergence in $L^2(\Omega)$\,. Now, by \eqref{eq:firstvars} and \eqref{banale}, using the very definition of $\widehat\J^{\Aa_n}$ in \eqref{defJshat}, we have that, for every $u,\ffi\in L^2(\Omega)$ 
\begin{equation*}
\lim_{t\to 0}\frac{\widehat\J^{\Aa_n}(u+t\ffi)-\widehat\J^{\Aa_n}(u)}{t}=\Big\langle-(-\Delta)^{-\frac{\Aa_n}{2}}u+2\frac{d\omega_d}{\Aa_n}u,\ffi\Big\rangle_{L^2(\Omega)}\,.
\end{equation*} 
Hence, the stability claim follows by applying Theorem \ref{genstab} with $\h=L^2(\Omega)$\,, $\F^n=\widehat\J^{\Aa_n}$ and $\F^{\infty}=\widehat\J^{0}$\,, once noticed that, in view of Theorem \ref{thm:stozerofo}, assumption (b) is satisfied.
\end{proof}
\subsection{Convergence of the $\Aa$-parabolic flows as $\Aa\to d^-$}
Here we state and prove the  convergence results for the parabolic flows corresponding to the (either scaled or renormalized) $\Aa$-Riesz functionals as $\Aa\to d^-$\,. Such results follow by Theorem \ref{genstab}.
 
\begin{theorem}\label{convheat0ordsto-d2}
Let $\{\Aa_n\}_{n\in\N}\subset (0,d)$ be such that $\Aa_n\to d^-$ as $n\to +\infty$\,. Let $u^{\infty}_0\in L^2(\Omega)$ and $\{u_0^n\}_{n\in\N}\subset L^2(\Omega)$ be such that and 
 $u^n_0\to u^{\infty}_0$ in $L^2(\Omega)$\,. Let $T>0$\,;
then,
for every $n\in\N$ there exists a unique  solution $u^n\in H^1([0,T]; L^2(\Omega))$ of
\begin{equation}\label{cauchyord0n-d2}
\begin{dcases}
 u_t(t) =\pm(-\Delta)^{-\frac{\Aa_n}{2}}u(t) \qquad\textrm{for a.e.\ }t\in(0,T)\\
u(0)=u^{n}_0\,.
\end{dcases}
\end{equation}
Moreover, $u^n\to u^{\infty}$ in $H^1([0,T];L^2(\Omega))$ as $n\to +\infty$\,, where $u^{\infty}$ is the unique 
solution of
\begin{equation}\label{cauchyord0infty-d2}
\begin{dcases}
 u_t(t) =\pm2\int_{\Omega}u(t)\ud y\qquad\textrm{for a.e.\ }t\in (0,T)\,,\\
u(0)=u^{\infty}_0\,,
\end{dcases}
\end{equation}
and
$$
\|u^n(t)-u^\infty(t)\|_{L^2(\Omega)}\to 0\quad\textrm{and}\quad \J^{\Aa_n}(u^n(t))\to \J^{d}(u^\infty(t))\qquad\textrm{for every }t\in[0,T]\,.
$$
\end{theorem}
\begin{proof}
By \eqref{welldefJs} the functionals $\pm\J^{\Aa_n}$ are finite on $L^2(\Omega)$ and, by Remark \ref{quadrpos}, they are also $\lambda$-positive and $\lambda$-convex for every $\lambda>2\omega_d\big(\mathrm{diam}(\Omega)\big)^d$\,.
Furthermore, by \eqref{continJs}, the functionals $\J^{\Aa_n}$ are continuous with respect to the strong convergence in $L^2(\Omega)$\,. 
 Moreover, for every $u,\ffi\in L^2(\Omega)$\,,
 \begin{equation}\label{banale-d2}
 \lim_{t\to 0}\frac{\J^{d}(u+t\ffi)-\J^{d}(u)}{t}=-\big\langle2\int_{\Omega}u\ud y,\ffi\big\rangle_{L^2(\Omega)}\,.
 \end{equation}
Finally, by Theorem \ref{thm:sto-d2}, we have that $\lim_{n\to +\infty}\J^{\alpha}(u^n_0)=\J^{\infty}(u^\infty_0)$\,.
Hence, the conclusion follow by applying Theorem \ref{genstab} with $\F^n =\pm\J^{\Aa_n}$ and $\F^{\infty}=\pm\J^{d}$, once noticed that, in view of Theorem \ref{thm:sto-d2}(a), assumption (a)  is satisfied. 
\end{proof}
\begin{theorem}\label{convheat1ordsto-d2}
Let $\{\Aa_n\}_{n\in\N}\subset (0,d)$ be such that $\Aa_n\to d^-$ as $n\to +\infty$\,. Let $u^\infty_0\in L^2(\Omega)$ and let $\{u^{n}_0\}_{n\in\N}\subset L^2(\Omega)$ be such that 
 $u^n_0\to u^\infty_0$ in $L^2(\Omega)$\,. Let $T>0$\,; then, 
for every $n\in\N$ there exists a unique  solution $u^n\in H^1([0,T]; L^2(\Omega))$ to
\begin{equation}\label{cauchyord1n-d2}
\begin{dcases}
 u_t(t) =\pm\frac{1}{d-\Aa_n}\Big[(-\Delta)^{-\frac{\Aa_n}{2}}u(t)-2\int_{\Omega}u(t)\ud y \Big] \qquad\textrm{for a.e.\ }t\in (0,T)\,,\\
u(0)=u^{n}_0\,.
\end{dcases}
\end{equation}
Moreover,  $u^n\to u^\infty$ in $H^1([0,T];L^2(\Omega))$ as $n\to +\infty$\,, where $u^\infty\in H^1([0,T];L^2(\Omega))$ is the unique solution to
\begin{equation}\label{cauchyord1infty-d2}
\begin{dcases}
 u_t(t) =\pm(-\Delta)^{-\frac d 2} u(t)\qquad\textrm{for a.e.\ }t\in (0,T)\\
u(0)=u^\infty_0\,,
\end{dcases}
\end{equation}
and
$$
\|u^n(t)-u^\infty(t)\|_{L^2(\Omega)}\to 0\quad\textrm{and}\quad \widetilde{\J}^{\Aa_n}(u^n(t))\to \widetilde{\J}^{d}(u^\infty(t))\qquad\textrm{for every }t\in[0,T]\,.
$$
\end{theorem}
\begin{proof}
By \eqref{ovvia}, \eqref{convespe}, and \eqref{lausodopo}, using Remark \ref{quadrpos}, we have that
the functionals $\pm\widetilde\J^d$ and
 $\pm\widetilde\J^{\Aa_n}$ are $\lambda$-convex and $\lambda$-positive for every $\lambda>4\|\tilde k^{d}\|_{L^1(B_{\mathrm{diam}(\Omega)}(0))}$\,. Moreover, by \eqref{continJs}, the functionals  $\widetilde\J^{\Aa_n}$  are continuous with respect to the strong convergence in $L^2(\Omega)$\,, and, by arguing as in \eqref{continJs} one can show that also the functional  $\widetilde\J^d$ is continuous with respect to the strong convergence in $L^2(\Omega)$\,. Now, by \eqref{eq:firstvars} and \eqref{banale-d2}, we have that, for every $u,\ffi\in L^2(\Omega)$ 
\begin{equation*}
\lim_{t\to 0}\frac{\widetilde\J^{\Aa_n}(u+t\ffi)-\widetilde\J^{\Aa_n}(u)}{t}=\frac{1}{d-\Aa_n}\Big\langle-(-\Delta)^{-\frac{\Aa_n}{2}} u+2\int_{\Omega}u(t)\ud y,\ffi\Big\rangle_{L^2(\Omega)}\,.
\end{equation*} 
Finally, by Theorem  \ref{thm:sto-d2fo}, $\lim_{n\to +\infty}\widetilde\J^{\Aa_n}(u^n_0)=\widetilde\J^{d}(u^{\infty}_0)$\,.
Hence, the conclusion follows by applying Theorem \ref{genstab} with $\F^n=\pm\widetilde\J^{\Aa_n}$ and $\F^{\infty}=\pm\widetilde\J^{d}$\,, once noticed that, in view of Theorem \ref{thm:sto-d2fo}, assumption (a) is satisfied.
\end{proof}


\begin{thebibliography}{99}
\bibitem{ADPM}
{L. Ambrosio, G. De Philippis, L. Martinazzi}: Gamma-convergence of nonlocal perimeter functionals. {\it Manuscripta Math. } {\bf 134} (2011), 377--403. 
\bibitem{AGS}
{L. Ambrosio, N. Gigli, G. Savar\'e}: {\it Gradient Flows in Metric Spaces and in the Space of Probability Measures}. Lectures in Mathematics ETH Z\"urich, Birkh\"auser Verlag (2005).
\bibitem{BBM1}
{J. Bourgain, H. Brezis, P. Mironescu}: Another look at Sobolev spaces. In {\it Optimal Control and
Partial Differential Equations} (J. L. Menaldi, E. Rofman and A. Sulem, eds.), a volume in honor
of A. Bensoussan's 60th birthday, IOS Press, 2001, 439--455.
\bibitem{BBM2}
{J. Bourgain, H. Brezis, P. Mironescu}: Limiting embedding theorems for $W^{s,p}$ when $s\uparrow 1$  and applications. {\it J. Anal. Math.} {\bf 87} (2002), 77--101.
\bibitem{BPS}
{L. Brasco, E. Parini, M. Squassina}: Stability of variational eigenvalues for the fractional $p$-Laplacian. {\it Discr. Contin. Dyn. Syst.}  {\bf 36} (2016), 1813--1845. 
\bibitem{Bre}
{H. Brezis}: {\it Op\'erateurs Maximaux Monotones et S\'emi-groupes de Contractions dans
les Espaces de Hilbert}. North Holland, Amsterdam, 1973.
\bibitem{CDNP}
{A. Cesaroni, L. De Luca, M. Novaga, M. Ponsiglione}: Stability results for nonlocal geometric evolutions and limit cases for fractional mean curvature flows. {\it Comm. Partial  Differ. Equ.} {\bf 46} (2021), 1344--1371.
\bibitem{CW}
{H. Chen, T. Weth}: The Dirichlet problem for the logarithmic Laplacian. {\it Comm. Partial  Differ. Equ.} {\bf 44} (2019), 1100--1139.
\bibitem{CDKNP}
{V. Crismale, L. De Luca, A. Kubin, A. Ninno, M. Ponsiglione}: The variational approach to $s$-fractional heat flows and the limit cases $s\to 0^+$ and $s\to 1^-$\,. {\it J. Funct. Anal.} {284} (2023), art. n. 109851.  
\bibitem{D}
{J. D\'avila}: {On an open question about functions of bounded variation}. {\it Calc. Var. Partial Differ. Equ.} {\bf 15} (2002), 519--527. 
\bibitem{DeGMT}
{E. De Giorgi, A. Marino, M. Tosques}: Problems of evolution in metric spaces
and maximal decreasing curve. {\it Atti Accad. Naz. Lincei Rend. Cl. Sci. Fis. Mat. Natur.}
{\bf 68} (1980), 180--187.
\bibitem{DKP}
{L. De Luca, A. Kubin, M. Ponsiglione}:The core-radius approach to supercritical fractional perimeters, curvatures and geometric flows. {\it Nonlinear Anal.} {\bf 214} (2022), art. n. 112585.
\bibitem{DNP}
{L. De Luca, M. Novaga, M. Ponsiglione}: {The $0$-fractional perimeter between fractional perimeters and Riesz potentials}. {\it Ann. SNS Pisa Cl. Sci.} {\bf XXII} (2021), 1559--1596.
\bibitem{DFPV}
{S. Dipierro, A. Figalli, G. Palatucci, E. Valdinoci}: Asymptotics of the s-perimeter as $s\searrow 0$. {\it Discrete Cont. Dyn. Syst.}
{\bf 33} (2013), 2777--2790.
\bibitem{DPV}
{E. Di Nezza, G. Palatucci, E. Valdinoci}: Hitchhiker's guide to the fractional Sobolev spaces. {\it Bulletin Sci. Math.} {\bf 136} (2012), 521--573.
\bibitem{JW}
{S. Jarohs, T. Weth}: Local compactness and nonvanishing for weakly singular nonlocal
quadratic forms. {\it Nonlinear Anal.} {\bf 193} (2020), 111431.
\bibitem{Landkof}
{N.S. Landkof}: {\it Foundations of Modern Potential Theory.} Springer, Berlin, 1972.
\bibitem{LS}
{G. Leoni, D. Spector}: Characterization of Sobolev and BV spaces. {\it J. Funct. Anal.} {\bf 261} (2011), 2926--2958.
\bibitem{MRT}
{J.M. M\'azon, J.D. Rossi, J. Toledo}: Fractional p-Laplacian evolution equations. {\it J. Math. Pures Appl.} {\bf 105} (2016), 810--844.
\bibitem{MS}
{V. Maz'ya, T. Shaposhnikova}: On the Bourgain,Brezis, and Mironescu Theorem Concerning
Limiting Embeddings of Fractional Sobolev Spaces. {\it J. Funct. Anal.} {\bf195} (2002), 230--238.
\bibitem{NSV}
{R.H. Nochetto, G. Savar\'e, C. Verdi}: A posteriori error estimates for variable time-step discretizations of nonlinear evolution equations. {\it Comm. Pure Appl. Math.} {\bf 53} (2000) 525--589.
\bibitem{O1}
{C. Ortner}: Two Variational Techniques for the Approximation of Curves of Maximal slope. Technical report NA05/10, Oxford University Computing Laboratory, Oxford, UK (2005).
\bibitem{O2}
{C. Ortner}: Gradient flows as a selection procedure for equilibria of nonconvex energies. {\it SIAM J. Math. Anal.} {\bf 38} (2006), 1214--1234.
\bibitem{P}
{A. Ponce}: A new approach to Sobolev spaces and connections to $\Gamma$-convergence. {\it Calc. Var. Partial Differ. Equ.} {\bf 19} (2004), 229--255.
\bibitem{SS}
{E. Sandier, S. Serfaty}: Gamma-Convergence of Gradient Flows with Applications to Ginzburg Landau. {\it Comm. Pure Appl. Math.} {\bf 57} (2004), 1627--1672.
\bibitem{Stein}
{E.M. Stein}: {\it Singular integrals and differentiability properties of functions.} Princeton University Press, 1970.
\bibitem{Vaz17}
{J.L. V\'azquez}: Asymptotic behaviour for the fractional heat equation in the Euclidean space. {\it Complex Var. Elliptic Equ.} {\bf 63} (2018), 1216--1231.
\end{thebibliography}
\end{document}